\documentclass{amsart}
\usepackage{amsthm,amsfonts,amsmath,amscd,amssymb,latexsym,epsfig,mathrsfs}
\usepackage{appendix}
\usepackage{tikz-cd}
\usepackage{xcolor}
\usepackage[pdftex,pdfpagelabels,bookmarks,hyperindex,hyperfigures]{hyperref}
\hypersetup{
    colorlinks = true,
    linkcolor = {blue},
    citecolor={magenta},
   linkbordercolor = {white},
    linktocpage=true
}
\usepackage[non-sorted-cites]{amsrefs}

\newcommand{\cx}{{\mathbb C}}

\newcommand{\Spec}{\operatorname{Spec}}

\numberwithin{equation}{section}

\newtheorem{theorem}{Theorem}[section]
\newtheorem*{theorem*}{Theorem}

\newtheorem{lemma}[theorem]{Lemma}

\newtheorem{corollary}[theorem]{Corollary}

\newtheorem{proposition}[theorem]{Proposition}

\theoremstyle{remark}

\newtheorem{definition}[theorem]{Definition}

\newtheorem{example}[theorem]{Example}

\newcommand{\C}{{\mathbb{C}}}

\newcommand{\oN}{{\mathbb{N}}}

\newcommand{\oP}{{\mathbb{P}}}

\newcommand{\oR}{{\mathbb{R}}}

\newcommand{\sE}{{\mathcal{E}}}

\newcommand{\sM}{{\mathcal{M}}}   

\newcommand{\sQ}{{\mathcal{Q}}}

\newcommand{\sV}{{\mathcal{V}}}

\newcommand{\fK}{{\mathfrak{k}}}

\newcommand{\fU}{{\mathfrak{u}}}

\begin{document}

\title{Hyperquiver varieties}
\author{Roger Bielawski}
\address{Institut f\"ur Differentialgeometrie,
Leibniz Universit\"at Hannover,
Welfengarten 1, 30167 Hannover, Germany}



\begin{abstract} We discuss the recently introduced concept of hyperquivers, due to Muller, Nanda, and Seigal, from the point of view of K\"ahler and hyperk\"ahler quotients. As examples, we discuss the projective K\"ahler hyperquiver varietes $\sM$, corresponding to hyperquivers with a single hyperedge.
\end{abstract}

\maketitle

\thispagestyle{empty}

\section{Hyperquivers}

In a recent paper \cite{MNS} the notion of a quiver and its representations has been generalised to {\em hyperquivers}.
Briefly, a hyperquiver consists of two finite sets $\sV$ (vertices) and $\sE$ (hyperedges) together with {\em multivalued} maps $s,t:\sE\to \sV$ (sources and targets). A {\em representation} of a hyperquiver associates to each vertex $v$ a vector space $V(v)$ and to each hyperedge with sources $s_1,\dots,s_k\in \sV$ and targets $t_1,\dots,t_l\in \sV$, a tensor 
$$T\in V(s_1)^\ast\otimes\dots\otimes V(s_k)^\ast \otimes V(t_1)\otimes\dots\otimes  V(t_l).$$
The visual representation of hyperquivers in \cite{MNS} is somewhat unsatisfactory: the only way to distinguish between: a) the ordinary quiver with three vertices $\{1,2,3\}$ and two edges $(1\to 2)$, $(1\to 3)$, and b) the hyperquiver with with three vertices $\{1,2,3\}$ and one hyperedge $e$ with $s(e)=\{1\}$, $t(e)=\{2,3\}$, is to make the edges in a) of different colours. We therefore propose an alternative description in  terms of $2$-dimensional cobordisms.
\par
We define first an $n$-dimensional {\em manifold with multiboundary} as a (second-countable) Hausdorff topological space which is locally homeomorphic either to $\oR^n$ or to $\oR^{n-1}\times B_k$, where $B_k$ is the union of positive coordinate half-axes in some $\oR^k$, i.e.\ $B_k=\{(y_1,\dots,y_k)\in [0,\infty)^k; \text{at most one $y_i\neq 0$}\}$. 
A transition function is required to be smooth on each $\oR^{n-1}\times \{(0,\dots,0,y_i,0,\dots,0)\}$.
\par
The second local model can be viewed as the union of $k$ copies of $\oR^{n-1}\times [0,\infty)$,  the boundaries of which have been identified with each other. We therefore call $k$ the {\em multiplicity of a boundary component}.  
\par
We now define a hyperquiver as consisting of a disjoint union $\sV$ of finitely many oriented circles (``vertices") and a finite set $\sE$ of oriented $2$-dimensional surfaces with multiboundaries in $\sV$ (``hyperedges"). Here ``oriented" means oriented in the interior. In a neighbourhood of a boundary circle, a hyperedge looks like a number of tubes attached to that circle. For each of these tubes the circle is either  source or  target: it is the source (resp.\ target) if the orientation given by the circle and the inward (resp.\ outward) pointing normal vector (in that order) agrees with the given orientation of the interior of the tube. It is straightforward to verify that this definition is equivalent to the original one.
\begin{example} The following are three hyperquivers with one vertex and one hyperloop, giving three different representation spaces. 
\begin{figure}[h]
	\includegraphics[width=0.6\textwidth]{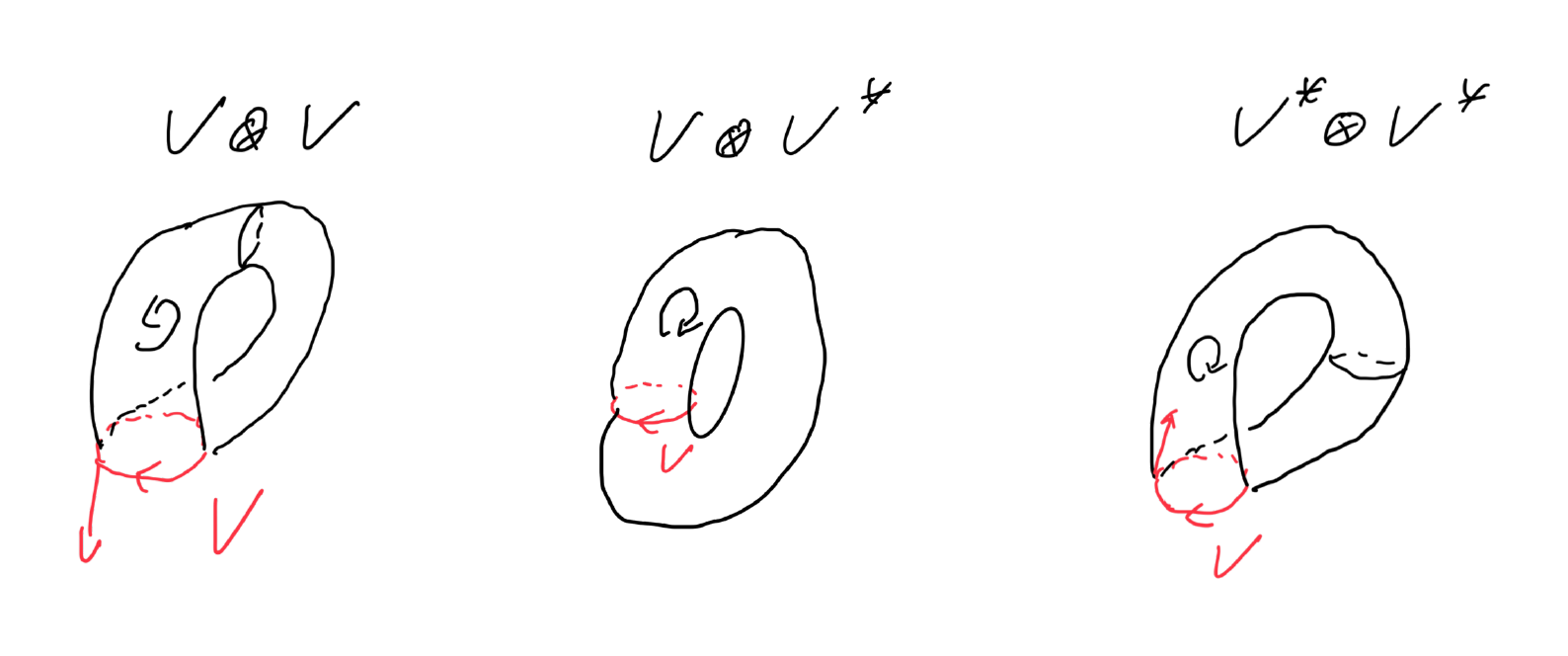}
\end{figure}

Only the middle hyperquiver is a quiver. The other two hyperquivers should be embedded into a higher-dimensional space, in order that the hyperedge does not have self-intersections outside the boundary. 
\label{three}\end{example}

\section{K\"ahler and hyperk\"ahler hyperquiver varieties}
Let $\sQ=(\sV,\sE)$ be a hyperquiver. To any  finite-dimensional  $\sV$-graded complex vector space $V=\bigoplus_{i\in \sV} V_i$ we associate the corresponding representation space
$$R(\sQ,V)=\bigoplus_{e\in \sE}\Bigl( \bigotimes_{i\in s(e)} V_i^\ast\otimes  \bigotimes_{j\in t(e)} V_j\Bigr).$$
The order of the tensor products in $ \bigotimes_{i\in s(e)} V_i^\ast$ and in $ \bigotimes_{j\in t(e)} V_j$ is given by  ordering  the elements of $\sV$. The representation spaces $R(\sQ,V)$ for different orderings are isomorphic to each other.
The group ${\rm GL}(V)=\prod_{i\in \sV}{\rm GL}(V_i)$ acts naturally on  $R(\sQ,V)$.
\par
Suppose now that $V$ is a hermitian vector space ($\sV$-graded, i.e.\ the summands $V_i$, $i\in \sV$, are mutually orthogonal). Then so is $R(\sQ,V)$ and we can form K\"ahler quotients of $R(\sQ,V)$ by the unitary group ${\rm U}(V)=\prod_{i\in \sV}{\rm U}(V_i)$. The moment map for the action of ${\rm U}(V_i)$ is given by
\begin{equation}  T=\sum_{e\in \sE}T_e\stackrel{\mu}{\longmapsto} \frac{\sqrt{-1}}{2} \sum_{\{e; t(e)=i\}} c^i(T_e,T_e)-\,\frac{\sqrt{-1}}{2}\sum_{\{e; s(e)=i\}} c^i(T_e,T_e),\label{moment}\end{equation}
where $c^i$ is the contraction of all indices except $i$ using the hermitian metric. If we choose an orthonormal basis $\{e_l^j\}$ for each $V_j$, and use the dual basis for each $V_j^\ast$, we can write an element $T$ of $W_1\otimes\dots \otimes W_r$, where each $W_s\in \{V_j,V_j^\ast\}$, as a $k\times l$ matrix $A$,  $k=\dim W_{s}$ and $l=\prod_{z\neq s}\dim W_{z}$, and then $c^{s}(T,T)=AA^\ast$.
\par
For any element $\theta$ of the centre of $\fU(V)=\bigoplus_{i\in \sV}\fU(V_i)$, we obtain the corresponding {\em K\"ahler hyperquiver variety} $\mu^{-1}(\theta)/{\rm U}(V)$ (which might be empty). It is a stratified K\"ahler manifold \cite{LS} and also a quasi-projective variety obtained as a good quotient of the locus of semistable points by ${\rm GL}(V)$. Here a point is ccalled semistable (resp.\ polystable) if the closure of its ${\rm GL}(V)$-orbit (resp.\ its ${\rm GL}(V)$-orbit) intersects $\mu^{-1}(\theta)$.
For a detailed description of GIT quotients of vector spaces the reader may consult \cite{Hos}.
\par
We remark that a general result of Boutot \cite{Bou} (see also \cite{Greb} for the analytic version) implies that K\"ahler hyperquiver varieties have rational singularities. In particular, they are normal and Cohen-Macaulay.
\par
Given a hyperquiver $\sQ=(\sV,\sE)$, we can consider, analogously to quivers, the corresponding {\em double hyperquiver} $\tilde\sQ=(\sV,\sE\sqcup \bar\sE)$, where the hyperedges in $\bar\sE$ are the hyperedges in $E$ with the opposite orientation. The representation space  $R(\tilde\sQ,V)$ is canonically identified with $T^\ast R(\sQ,V)$, and a choice of a hermitian metric on $V$ makes it into a quaternion-hermitian vector space. We can then form hyperk\"ahler quotients of $R(\tilde\sQ,V)$ by ${\rm U}(V)$ and obtain {\em hyperk\"ahler hyperquiver varieties}, which are stratified hyperk\"ahler manifolds \cites{DS,Mayr}. We remark that hyperk\"ahler quotients of quaternionic representations of cotangent type are discussed in detail in \cite{Aba}.

\section{Moduli spaces of tensors}

We shall now investigate K\"ahler  hyperquiver varieties corresponding to perhaps the simplest  hyperquiver, with one hyperedge and $m$ vertices: 
\begin{figure}[h]
	\includegraphics[width=0.6\textwidth]{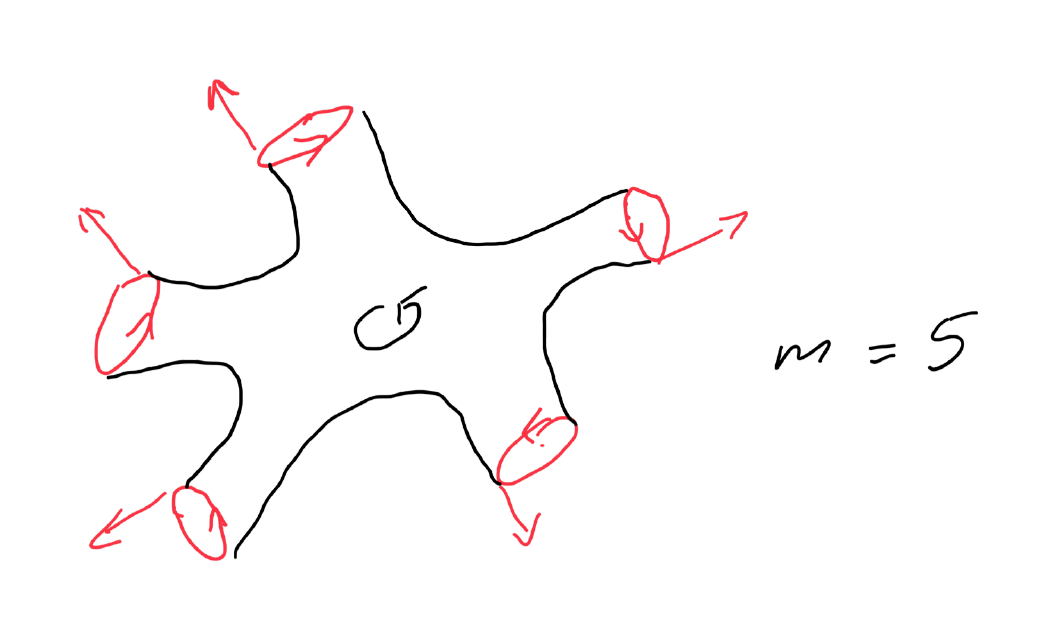}
\end{figure}

We choose the orientations of the hyperedge and of the boundary circles in such way that all vertices are targets. To each vertex we associate $\cx^{k_i}$, $k_i\in \oN$, with its standard hermitian inner product. 
The representation space is then the tensor product of these vector spaces: 
$$R(\sQ,V)=\cx^{k_1}\otimes\dots\otimes \cx^{k_m},$$
and we denote by $\sM_{\vec{k}}$, where $\vec{k}=(k_1,\dots,k_m)$,  the corresponding K\"ahler hyperquiver variety defined in the previous section. Without loss of generality we may assume that each $k_i\ge 2$.
\par
The  torus  $T^{m-1}=\{(t_1,\dots,t_m)\in U(1)^m;t_1\cdots t_m=1\}$ acts trivially on $R(\sQ,V)$ and can be viewed as a subgroup of the centre of $U(V)$. The compact group $K=\prod_{i=1}^m {\rm U}(k_i)/T^{m-1}$ acts effectively on  $R(\sQ,V)$. The Lie algebra $\fK$ of $K$ can be identified with elements $(u_1,\dots,u_m)$ of $\bigoplus_{i=1}^m\fU(k_i)$ such that the central components of the $u_i$, viewed as scalars, are all equal to each other. We consider the K\"ahler quotient $\mu^{-1}(I)/K$ of $R(\sQ,V)$, where $I=\frac{\sqrt{-1}}{2}(1,\dots,1)$. According to \eqref{moment}, the moment map equations for a tensor $T\in \bigotimes_{i=1}^m \cx^{k_i} $ are
\begin{equation} c^i(T,T)=1,\quad i=1,\dots, m.\label{moment2}\end{equation}
These equations imply that a nonempty $\sM_{\vec{k}}$ is compact, i.e.\ projective. In order to determine this variety, observe that we can perform the K\"ahler quotient in two steps: first by the group $\prod_{i=1}^m {\rm SU}(k_i)$, and then by the central circle in $K$. The complex structure of the first quotient is that of the {\em affine quotient} $\Spec\C[R(\sQ,V)]^{G_0}$, where $G_0=\prod_{i=1}^m {\rm SL}_{k_i}(\C)$. The induced action of $\cx^\ast$ on this variety is given by $\lambda.p=\lambda^{\deg p}p$ for any invariant homogeneous polynomial $p\in \C[R(\sQ,V)]^{G_0}$, and it makes $\C[R(\sQ,V)]^{G_0}$ into a (nonnegatively) graded ring.  
The following description of $\sM_{\vec{k}}$ follows easily from the GIT-theory:
\begin{proposition}  $\sM_{\vec{k}}$ is isomorphic to ${\rm Proj}\,\C[R(\sQ,V)]^{G_0}$, i.e.\ to the geometric quotient of $\Spec\C[R(\sQ,V)]^{G_0}\setminus\{0\}$ by the above $\C^\ast$-action.\qed\label{steps}
\end{proposition}
\begin{example}
The case of $\vec{k}=(3,3,3)$ has been studied in detail by Ng \cites{Ng1,Ng2}. In particular Ng has shown that the surface $\sM_{(3,3,3)}$ is rational. We can identify this surface explicitly, using the above proposition. It is namely known \cite{Vinb}, that the ring of ${\rm SL}_{3}(\C)$-invariants on $\cx^3\otimes\cx^3\otimes \cx^3$ is freely generated in degrees $6,9$, and $12$. Therefore 
$$\sM_{(3,3,3)}\simeq \oP^2(6,9,12)\simeq \oP^2(2,3,4)\simeq \oP^2(1,3,2).$$
\end{example}
Observe that equations \eqref{moment2} imply
 the row rank of the $k_i\times (\prod_{j\neq i}k_j)$ matrix $A_i$ defined after \eqref{moment} equals $k_i$. Thus the multilinear rank\footnote{The multilinear rank of a tensor $T\in V_1\otimes\dots\otimes V_m$ is the $m$-tuple of dimensions $\dim\{T(v); v\in V_i^\ast\}$.} of a tensor $T\in \mu^{-1}(I)$ must be equal to $(k_1,\dots,k_m)$. It is clear (cf.\ \cite[Ex. 2.4.2]{Land}) that the multilinear rank $(r_1,\dots,r_m)$ of a tensor in $\bigotimes_{i=1}^m \cx^{k_i} $ satisfies $r_i\leq \min(k_i,\prod_{j\neq i}k_j)$, $i=1,\dots,m$.
 Hence:
\begin{lemma} If $k_i> \prod_{j\neq i}k_j$ for some $i\in\{1,\dots,m\}$, then $\sM_{\vec{k}}$ is empty.\qed
\end{lemma}
We therefore adopt the following definition:
\begin{definition} A dimension vector $\vec{k}=(k_1,\dots,k_m)\in \oN^m$ is said to be {\em balanced}, if $k_i \leq \prod_{j\neq i}k_j$ for each $i=1,\dots,m$.  
\end{definition}
\begin{corollary} If $\vec{k}$ is unbalanced, then $0$ is the only closed orbit of $\prod_{i=1}^m{\rm SL}_{k_i}(\C)$ on $\bigotimes_{i=1}^m \cx^{k_i}$.
\end{corollary}
\begin{proof} Since $\sM_{\vec{k}}$ is empty, Proposition \ref{steps} implies that the the affine quotient $\Spec\C[R(\sQ,V)]^{G_0}$ is a point.
\end{proof}
For a balanced $\vec{k}$, we may ask whether $\sM_{\vec{k}}$ is rational. The answer is provided by:
\begin{proposition} Let $\vec{k}$ be balanced. The variety $\sM_{\vec{k}}$  is rational if and only if the generic stabiliser of $K^\C$ is finite. In this case: \begin{equation*} \dim_\C \sM_{\vec{k}}=\prod_{i=1}^m k_i-\sum_{i=1}^m k_i^2+(m-1).\label{dim}\end{equation*}
\end{proposition}
\begin{proof} The first statement follows from a theorem of Rosenlicht: see, e.g.\ the Corollary to Lemma 2.4 in \cite{PV}. The second statement is obvious.
\end{proof}
A sufficient condition is given by:
\begin{proposition} If the dimension vector  $\vec{k}=(k_1,\dots,k_m)\in \oN^m$ satisfies $2k_i < \prod_{j\neq i}k_j$ for each $i=1,\dots,m$, then the generic stabiliser of $K^\C$ acting on $R(\sQ,V)$ is finite. 
\end{proposition}
\begin{proof} We can equivalently ask that the generic stabiliser of $\prod_{i=1}^m {\rm SL}_{k_i}(\C)$ be finite. The stated condition follows easily from the main theorem in \cite{AVE}.
\end{proof}

\subsection{Other hyperquivers with a single hyperedge}
Consider a hyperquiver with a single edge and $m$ vertices with multiplicities $l_1,\dots,l_m$. For each of these vertices let $l_i^-$ (resp.\ $l_i^+)$) be the number of times the $i$-th vertex is a source (resp.\ a target).
Let $\vec{k}=(k_1,\dots,k_m)$ be the dimension vector. The corresponding representation space is then
\begin{equation} \bigotimes_{i=1}^m(\C^{k_i})^{\otimes l_i^+}\otimes  \bigl((\C^{k_i}\bigr)^\ast)^{\otimes l_i^-},\label{newQ}\end{equation}
and the hyperquiver variety is obtained as the K\"ahler quotient by (the non-effective) $\prod {\rm U(k_i)}$, with each ${\rm U}(k_i)$ acting diagonally on $(\C^{k_i})^{\otimes l_i^+}\otimes  \bigl((\C^{k_i}\bigr)^\ast)^{\otimes l_i^-}$.
\par
Let us choose again $\theta=I$ and suppose that $\mu(T)=I$, where $\mu$ is moment map \eqref{moment}. If we add together the traces of $-\sqrt{-1}\mu_i(T)$, where $\mu_i$ is the moment map for ${\rm U(k_i)}$, we shall obtain:
$$\|T\|^2 \sum_{i=1}^m (l_i^+-l_i^-)=\sum_{i=1}^m k_i$$
Therefore the hyperquiver variety is empty if $N= \sum_{i=1}^m (l_i^+-l_i^-)$ is $\leq 0$, and projective if $N>0$.
Its complex structure is given by the ${\rm Proj}$ of the subring of the ring of invariants of $\prod_{i=1}^m{\rm SL}_{k_i}(\C)$, generated by the constants and by invariant homogeneous polynomials of positive degree with respect to  the diagonal $\cx^\ast$-action.
\par
  We can also apply again the main result of \cite{AVE} to conclude that the generic stabiliser for the action of $\prod_{i=1}^m{\rm SL}_{k_i}(\C)$ on the representation space is finite, provided
$2k_i< \prod_{j\neq i}k_j^{l_j}$ whenever $l_i=1$.

\end{document}